\documentclass[11pt]{amsart}
\usepackage[margin=1.5in]{geometry}
\usepackage{cite}
\usepackage[utf8]{inputenc}
\usepackage[T1]{fontenc}
\usepackage[english]{babel}
\usepackage{amsmath,amsfonts,amsthm, mathrsfs}
\usepackage{hyperref}
\usepackage{blkarray}
\usepackage{chngpage}

\allowdisplaybreaks

\theoremstyle{plain}
\newtheorem{theorem}{Theorem}[section]

\newtheorem{lemma}[theorem]{Lemma}
\newtheorem{proposition}[theorem]{Proposition}

\theoremstyle{definition}

\newtheorem{example}[theorem]{Example}
\newtheorem{remark}[theorem]{Remark}
\numberwithin{equation}{section}

\DeclareMathOperator{\Prob}{P}

\title{The square and add {M}arkov chain}

\author{Persi Diaconis}
\address{Department of Mathematics, Stanford University, Stanford, CA  94305}
\email{diaconis@math.stanford.edu}

\author{Jimmy He}
\address{Department of Mathematics, Stanford University, Stanford, CA  94305}
\email{jimmyhe@stanford.edu}

\author{I. Martin Isaacs}
\address{Department of Mathematics, University of Wisconsin, Madison, WI  53706}
\email{isaacs@math.wisc.edu}

\dedicatory{In memory of John Conway.}

\begin{document}
\maketitle
\begin{abstract}
Squaring and adding $\pm 1\pmod{p}$ generates a curiously intractable random walk. A similar process over the finite field $\mathbf{F}_q$ (with $q=2^d$) leads to novel connections between elementary Galois theory and probability.
\end{abstract}

\section{Introduction}
Let us begin with a problem we cannot solve. If $q$ is a prime power, we write $\mathbf{F}_q$ to denote the field with $q$ elements, so if $p$ is prime, $\mathbf{F}_p$ is the integers modulo $p$. A simple random walk (drunkard's walk) on $\mathbf{F}_p$ goes from $j$ to $j+1$ or $j-1$ with probability $1/2$. As time goes on, this converges to the uniform distribution on $\mathbf{F}_p$. This means that after a long time, the probability that the random walk will be at some $\alpha\in\mathbf{F}_p$ is about $1/p$. It takes about $p^2$ steps for this convergence to kick in. This is slow---if $p=101$, $p^2=10,201$. These informal statements are explained more carefully after Theorem \ref{thm: main}.

One attempt to speed things up intersperses deterministic doubling with the random $\pm 1$ steps. If $X_n$ denotes the position of the walk after $n$ steps (say starting from $X_0=0$), this new walk is
\begin{equation*}
    X_n=2X_{n-1}+\varepsilon_n \pmod{p},
\end{equation*}
with $\varepsilon_n=\pm 1$ with probability $1/2$, independently from step to step.

In \cite{CDG87}, it is shown that order $\log(p)$ steps are necessary and sufficient for convergence ($\log$ will always refer to the natural logarithm). See \cite{EV20} for amazing refinements, and \cite{CD20} for other applications to deterministic speedup.

Seeking to understand such speedups, we consider the random walk
\begin{equation*}
    X_n=X_{n-1}^2+\varepsilon_n \pmod{p}.
\end{equation*}
This is the problem we cannot solve! We do not understand the stationary distribution---numerical evidence at the end of this paper shows it is wildly non-uniform. We do not even know its support, much less rates of convergence to stationarity.

Squaring defines an automorphism of a finite field of $2$-power order, so we decided to study the corresponding problem over the field $\mathbf{F}_q$, where $q = 2^d$. To be specific, we choose a basis $\mathcal{B}$ for $\mathbf{F}_q$ over its prime subfield $\mathbf{F}_2$, so $|\mathcal{B}| = d$, and we consider the random walk on the elements of $\mathbf{F}_q$ defined by setting $X_0 = 0$ and 
\begin{equation}
\label{eq: chain defn}
    X_n = X_{n-1}^2 + \epsilon_n
\end{equation}
for $n > 0$. Here, $\epsilon_n$ is randomly chosen from the set $\{0\} \cup \mathcal{B}$, where the probability that $\epsilon_n = 0$ is $1/2$, and for each element $\alpha \in \mathcal{B}$, the probability that $\epsilon_n = \alpha$ is $\frac{1}{2d}$. The unique stationary distribution for this walk is the uniform distribution $\pi(\alpha) = 1/2^d$. (Random walks, or in more formal language, Markov chains, are discussed in greater detail in Section \ref{sec: Markov chains}, below.)

If we were to omit the squaring, and simply take
$X_n = X_{n-1} + \epsilon _n$, it is not hard to see that the behavior of the resulting walk would be independent of the choice of the basis $\mathcal{B}$ that defines it. Surprisingly, however, the walk we defined above (which includes squaring) does depend on the choice of the basis. To illustrate this, we compute the transition matrices for the square-and-add Markov chains on $\mathbf{F}_8$ defined using two different bases. As we shall see, these matrices have different eigenvalues.

First, we explain what we mean by the ``transition matrix'' for a Markov chain on a finite set $X$. This is a square matrix $M$, with rows and columns indexed by the members of $X$, where for $x,y \in X$, the entry $M(x,y)$ in row $x$ and column $y$ is the probability of arriving at $y$ in one step, starting at $x$. 

To see the relevance of the transition matrix, write $v_n$ to denote the row vector having entries indexed by the elements of $X$, where the entry at position $x$ in $v_n$ is the probability that the random walk has arrived at $x$ at time $n$. It is easy to see that $v_{n+1} = v_n M$, so $v_n = v_0 M^n$, and thus the convergence of the Markov chain is controlled by the powers of the transition matrix $M$.

To compute transition matrices for our walks on $\mathbf{F}_8$, we need to name the elements of this field, and to do this, we take advantage of the fact that in general, the multiplicative group of the finite field $\mathbf{F}_q$ is cyclic of order $q-1$. If we fix a generator $r$ for this group, (so $r$ is a ``primitive element'') we see that the elements of the field are $0$ and $r^i$ for
$0 \leq i \leq q-2$.

Naming the field elements in this way, it is trivial to see how to compute the product of two field elements, but it is not clear how to determine their sum. In fact, more information is needed before this is possible: it suffices, for example, to know the minimal polynomial $f(x)$ of $r$ over the prime subfield of
$\mathbf{F}_q$. Taking $q = 2^d$, we see that $f$ is an irreducible polynomial of degree $d$ over $\mathbf{F}_2$, so if $q = 8$, we can assume that $f(x) = x^3 + x + 1$, and with this information, the arithmetic in $\mathbf{F}_8$ is completely determined. 

Perhaps this is an appropriate point to mention John Conway's significant contribution. We have seen that to do arithmetic computations in the finite field $\mathbf{F}_q$, where $q = p^d$ for some prime $p$, we need to choose an irreducible polynomial $f(x)$ of degree $d$ over $\mathbf{F}_p$, and we require that the roots of $f$ in $\mathbf{F}_q$ are primitive elements. For each prime power $q = p^d$, Conway defined an explicit polynomial with these properties. (For example, if
$q = 8$, the corresponding Conway polynomial is
$x^3 + x + 1$.) These Conway polynomials are used in computer software---for example Magma---to do computations in finite fields.   

The transition matrix $M$ for a Markov chain on $\mathbf{F}_8$ is an $8 \times 8$ matrix whose rows and columns are indexed by the field elements, and we choose to write these elements in the order: $0$, $1$, $r$, $r^2$, $r^3$, $r^4$, $r^5$, $r^6$, and we recall that the entry $M(\alpha,\beta)$ is the probability that one step of the chain goes from $\alpha$ to $\beta$. 

If we take the basis $\mathcal{B} = \{1, r, r^2\}$, it is not hard to compute that the matrix is
\begin{equation*}
\begin{blockarray}{ccccccccc}
&0&1&r&r^2&r^3&r^4&r^5&r^6\\
\begin{block}{c(cccccccc)}
0&\frac{1}{2} & \frac{1}{6} & \frac{1}{6} & \frac{1}{6} & 0 & 0 & 0 & 0 \\
1&\frac{1}{6} & \frac{1}{2} & 0 & 0 & \frac{1}{6} & 0 & 0 & \frac{1}{6} \\
r&\frac{1}{6} & 0 & 0 & \frac{1}{2} & 0 & \frac{1}{6} & 0 & \frac{1}{6} \\
r^2&0 & 0 & \frac{1}{6} & \frac{1}{6} & 0 & \frac{1}{2} & \frac{1}{6} & 0 \\
r^3&0 & \frac{1}{6} & 0 & \frac{1}{6} & 0 & 0 & \frac{1}{6} & \frac{1}{2} \\
r^4&\frac{1}{6} & 0 & \frac{1}{2}&0 & \frac{1}{6} & \frac{1}{6} & 0 & 0 \\
r^5&0 & \frac{1}{6} & \frac{1}{6} & 0 & \frac{1}{2} & 0 & \frac{1}{6} & 0 \\
r^6&0 & 0 & 0 & 0 & \frac{1}{6} & \frac{1}{6} & \frac{1}{2} & \frac{1}{6}\\
\end{block}
\end{blockarray}
\end{equation*}
The eigenvalues of this matrix are $0$, $0$, $0$, $2/3$, $1$ and the three cube roots of $4/27$. 

If instead we take $\mathcal{B} = \{r^3,r^5,r^6\}$, the transition matrix is
\begin{equation*}
\begin{blockarray}{ccccccccc}
&0&1&r&r^2&r^3&r^4&r^5&r^6\\
\begin{block}{c(cccccccc)}
0&\frac{1}{2} & 0 & 0 & 0 & \frac{1}{6} & 0 & \frac{1}{6} & \frac{1}{6} \\
1&0 & \frac{1}{2} & \frac{1}{6} & \frac{1}{6} & 0 & \frac{1}{6} & 0 & 0 \\
r&0 & \frac{1}{6} & 0 & \frac{1}{2} & \frac{1}{6} & 0 & \frac{1}{6} & 0 \\
r^2&0 & \frac{1}{6} & 0 & 0 & \frac{1}{6} & \frac{1}{2} & 0 & \frac{1}{6} \\
r^3&\frac{1}{6} & 0 & \frac{1}{6} & 0 & 0 & \frac{1}{6} & 0 & \frac{1}{2} \\
r^4&0 & \frac{1}{6} & \frac{1}{2}&0 & 0 & 0 & \frac{1}{6} & \frac{1}{6} \\
r^5&\frac{1}{6} & 0 & 0 & \frac{1}{6} & \frac{1}{2} & \frac{1}{6} & 0 & 0 \\
r^6&\frac{1}{6} & 0 & \frac{1}{6} & \frac{1}{6} & 0 & 0 & \frac{1}{2} & 0\\
\end{block}
\end{blockarray}
\end{equation*}
and the eigenvalues of this matrix are $0$, $1$, the three cube roots of $1/27$ and the three cube roots of $8/27$.

Since these two Markov chains on $\mathbf{F}_8$ have transition matrices with different sets of eigenvalues, we see that random walks determined by different bases for $\mathbf{F}_8$ can have different long-term behaviors. We do not know, however, the extent to which the choice of a basis for $\mathbf{F}_q$, can affect the rate of convergence of the corresponding Markov chain.

The second of our two bases for $\mathbf{F}_8$, namely $\{r^3,r^5,r^6\}$, consists of an orbit under the automorphism group of $\mathbf{F}_8$, which is the group generated by the squaring map. In fact, for every prime power $q$, there always exists a basis for $\mathbf{F}_q$ that forms an orbit under the automorphism group of the field. Such a basis is said to be a \emph{normal basis}, and it happens that our basis $\{r^3,r^5,r^6\}$ is the unique normal basis for $\mathbf{F}_8$. (The set $\{r,r^2,r^4\}$ is also an orbit under the automorphism group, but it is not a basis because $r + r^2 + r^4 = 0$ since $r$ is a root of the polynomial $x^3 + x + 1$.)

Although the properties of a Markov chain on $\mathbf{F}_q$ defined by choosing a basis can depend on the chosen basis, it can be proved that the transition matrices for chains defined by normal bases are identical up to an appropriate renaming of the field elements. It follows that the corresponding random walks are essentially the same. In fact, if $q=2^d$, it is not hard to show that after a multiple of $d$ steps, the probability distribution of a square-and-add walk defined on $\mathbf{F}_q$ by using a normal basis is exactly the same as the distribution for the walk on $\mathbf{F}_q$ without squaring. Also, this is the same as simple random walk on the binary hypercube, and it is well known that this walk takes $\frac{1}{2}d(\log(d)+c)$ steps to converge \cite{D88}.

There is one situation where a sharp analysis of the square-and-add Markov chain on a field of $2$-power order is possible. Following a suggestion of Amol Aggarwal, we let $p$ be a prime such that $2$ is a primitive root modulo $p$, which means that $2$ generates the multiplicative group of $\mathbf{F}_p$. (According to the Artin conjecture, these have positive density among all primes, and this can be proved assuming the generalized Riemann hypothesis.) 

Then, for $d = p - 1$, the cyclotomic polynomial
\begin{equation}
\label{eq: irred poly}
    f(x) = x^d + x^{d-1} + \cdots + x + 1
\end{equation}
is irreducible over $\mathbf{F}_2$. (These polynomials are discussed in Section \ref{sec: cyclotomic}.) With these assumptions, the field $\mathbf{F}_2[x]/(f)$ has order $2^d$, and a basis is 
\begin{equation}
\label{eq: basis}
   \{1,x,x^2,\ldots,x^{d-1}\}. 
\end{equation}
(Note that $x$ is not a primitive element of this field because $x^d = 1$, and so $x$ does not have order $2^d$.) The following result says roughly that about ${1 \over 2}d \log(d)$ steps are necessary and sufficient for convergence of the Markov chain determined by this basis on the field $\mathbf{F}_2[x]/(f)$.

If $K$ denotes the transition matrix for a Markov chain, let $K^n(\alpha,\beta)$ denote the probability of moving from $\alpha$ to $\beta$ in $n$ steps. Let
\begin{equation*}
    \|P-Q\|_{TV}=\frac{1}{2}\sum_{\alpha\in\mathbf{F}_q}|P(\alpha)-Q(\alpha)|
\end{equation*}
denote the total variation distance of probability measures.

\begin{theorem}
\label{thm: main}
Let $p$ be a prime with $2$ a primitive root in $\mathbf{F}_p$ and let $d=p-1$. In $\mathbf{F}_q$, with $q=2^d$, the Markov chain $\eqref{eq: chain defn}$, defined by the basis \eqref{eq: basis}, satisfies for $n=\frac{1}{2}d(\log(d)+c)$ with $c>0$,
\begin{equation*}
    \|K^n(0,\cdot)-\pi\|_{TV}\leq ae^{-bc},
\end{equation*}
and satisfies for $n=\frac{1}{2}d(\log(d)-c)$ with $c>0$,
\begin{equation*}
    \|K^n(0,\cdot)-\pi\|_{TV}\geq 1-a'e^{-bc}
\end{equation*}
for universal constants $a'$, $a$ and $b$, where $\pi$ denotes the uniform measure.
\end{theorem}
Informally, the precise upper and lower bounds in Theorem \ref{thm: main} can be phrased as ``about $\frac{1}{2}d\log(d)$ steps are necessary and sufficient for convergence''.

The heart of the proof is some magical combinatorics for the Frobenius map of repeated squaring. It is the kind of magic John Conway enjoyed.

\begin{remark}
Theorem \ref{thm: main} holds in more generality. As long as $d$ is even, $f$ defined by \ref{eq: irred poly} has no repeated factors, and so the random walk can be defined on the quotient $\mathbf{F}_2[x]/(f)$ which will be a direct sum of fields (with componentwise addition and multiplication). Squaring is still an isomorphism in this case. This is proved in Lemma \ref{lem: cyclotomic irr}. The same bounds hold in this case. Theorem \ref{thm: main} can also be extended (although with weaker estimates) to general primes $p$, with the random walk being $X_{n+1}=X_n^p+\varepsilon_{n+1}$.
\end{remark}

The combinatorics of combining adding and multiplying in finite fields is currently a hot topic in additive combinatorics, see \cite{G05}. The problems studied here seem different.

\subsection{Outline}
Section \ref{sec: background} contains background material on Markov chains, finite fields, and Fourier analysis over $\left(\mathbf{F}_2\right)^d$. Theorem \ref{thm: main} is proved in Section \ref{sec: proof}. Section \ref{sec: over F_p} returns to the square-and-add walk over $\mathbf{F}_p$ and has some computed examples and open questions.

\section{Background}
\label{sec: background}
This section contains some needed background on Markov chains, finite fields, and on Fourier analysis over $\left(\mathbf{F}_2\right)^d$. It presents these topics in a form needed to prove Theorem \ref{thm: main}.

\subsection{Markov chains}
\label{sec: Markov chains}
A \emph{Markov chain} is a sequence of random variables $X_n$ taking values in some finite set $X$, so that $X_{n+1}$ depends on $X_1,\dotsc, X_n$ solely through $X_n$. We will assume that our Markov chains are \emph{homogeneous}, which means that the chance of moving from one state to another at step $n$ doesn't depend on $n$. Such a process can be represented using a matrix $P$ indexed by $X$, whose entries $P(x,y)$ encode the chance of moving from $x$ to $y$. Here, by convention probability distributions are written as row vectors, and $P$ acts on the right, so if $\mu_n(x)$ is the chance of being at $x$ after $n$ steps of the Markov chain, then $\mu_n=\mu_{n-1}P$.

A \emph{stationary distribution} for the Markov chain defined by $P$ is some probability measure $\pi$ on $X$ so that $\pi P=\pi$. A Markov chain is said to be \emph{irreducible} if for any two states $x,y\in X$, there is some positive integer $t$ such that $P^t(x,y)>0$. This means that it is possible to reach any state from any other in the chain. A Markov chain is said to be \emph{aperiodic} if $P^t(x,x)>0$ for all sufficiently large $t$. Note that a sufficient condition for $P$ to be aperiodic is for $P^s(x,x)>0$ and $P^t(x,x)>0$ for some $s,t\geq 1$ with $(s,t)=1$. By the Perron--Frobenius theorem, an aperiodic, irreducible Markov chain has a unique stationary distribution.

\subsection{Finite fields}
This classical subject is exhaustively developed in \cite{LN97}. Throughout, we take $q=p^d$, where $p$ is prime, and we write $\mathbf{F}_q$ to denote the unique field with $q$ elements. If $f$ is an arbitrary irreducible degree $d$ polynomial with coefficients in $\mathbf{F}_p$, then
\begin{equation*}
    \mathbf{F}_q\cong \mathbf{F}_p[x]/(f),
\end{equation*}
and if we represent $\mathbf{F}_q$ in this way, we see that the set $\{1,x,x^2,\dotsc, x^{d-1}\}$ is a basis for $\mathbf{F}_q$ over its prime subfield $\mathbf{F}_p$.

Even if $f(x)$ is \emph{not} irreducible, $\mathbf{F}_p[x]/(f)$ is still an algebra over $\mathbf{F}_p$ and $1,x,\dotsc, x^{d-1}$ is still a basis. This algebra is readily identified provided that $f$ has no repeated factors.

\begin{lemma}
Let $f(x)\in\mathbf{F}_p[x]$ have no repeated factors. Suppose that $f=\prod f_i$ where the degree of $f_i$ is $d_i$. Then
\begin{enumerate}
    \item $\mathbf{F}_p[x]/(f)$ is isomorphic to the direct sum of the fields $\mathbf{F}_p[x]/(f_i)\cong \mathbf{F}_{p^{d_i}}$.
    \item The map $y\mapsto y^p$ is an automorphism on $\mathbf{F}_p[x]/(f)$.
\end{enumerate}
\end{lemma}
\begin{proof}
The first claim is a restatement of the Chinese remainder theorem, and the second claim follows from the first, since the map $y\mapsto y^p$ is an automorphism for each factor.
\end{proof}

The random walk \eqref{eq: chain defn} can be defined on the algebra $\mathbf{F}_p[x]/(f)$, using the basis $\mathcal{B}=\{1,x,x^2,\dotsc, x^{d-1}\}$, even if the polynomial $f$ is not irreducible, and provided that $f$ has no repeated factors, this walk has a uniform stationary distribution. In the following lemma, we take $p=2$.
\begin{lemma}
Let $f\in\mathbf{F}_2[x]$, where $f$ has no repeated factors. Then the Markov chain on $\mathbf{F}_2[x]/(f)$ defined as in \eqref{eq: chain defn} with respect to the basis $\mathcal{B}=\{1,x,x^2,\dotsc, x^{d-1}\}$ is irreducible, aperiodic, and has a unique stationary distribution, which is uniform.
\end{lemma}
\begin{proof}
Factor the transition matrix for the random walk as $K=PT$, where $T$ is the transition matrix for the walk defined by $X_n=X_{n-1}+\varepsilon_n$ and $P$ is the permutation matrix encoding the bijection $y\mapsto y^2$ on $\mathbf{F}_2[x]/(f)$.

Since $P$ is a permutation matrix, it has some finite order, so $P^n=I$ for some $n>0$. First, we show that $K^n(\alpha,\beta)>0$ if $T(\alpha,\beta)>0$. It will be useful to view a step from $K$ as applying $P$ followed by a step from $T$. Since $T$ is lazy, we can always apply $P$, and then remain stationary for the step from $T$, so do this $n-1$ times. At the very last step, instead of remaining stationary take a step from $T$. The result is moving according to $P$ exactly $n$ times, returning to the initial state, and then a step from $T$, and so $K^n(\alpha,\beta)>0$ if $T(\alpha, \beta)>0$. 

To see that $K$ is irreducible, observe first that $T$ is irreducible, so there exists a path using steps from $T$ that goes from $\alpha$ to $\beta$. Since each step of $T$ can be mimicked by a block of $n$ steps from $K$, it follows that there is a path from $\alpha$ to $\beta$ using steps from $K$.

To see that the Markov chain is aperiodic, start by taking a single step from $K$, say going from $\alpha$ to $\beta$, and then take steps in blocks of size $n$, going from $\beta$ back to $\alpha$, which is possible since $T$ is irreducible and $K^n(\alpha,\beta)>0$ if $T(\alpha, \beta)>0$. This means $K^{kn+1}(\alpha,\alpha)>0$ for some $k$. Also, since $T$ is lazy, $K^n(\alpha,\alpha)>0$. Because $kn+1$ and $n$ are coprime, the Markov chain is aperiodic.

Finally, since $T$ and $P$ both preserve the uniform distribution, so does $K$. Irreducibility and aperiodicity imply uniqueness of the stationary distribution.
\end{proof}

\subsubsection{Cyclotomic polynomials}
\label{sec: cyclotomic}
Fix $n\in\mathbf{N}$ and let the \emph{cyclotomic polynomials} $\Phi_n(x)\in\mathbf{Z}[x]$ be defined by
\begin{equation*}
    \Phi_n(x)=\prod_{\substack{1\leq k\leq n\\ \gcd(k,n)=1}}(x-e^{2\pi i k/n}).
\end{equation*}
The following facts are well-known (see \cite[Chapter 2 \S 4]{LN97} for example):
\begin{itemize}
    \item $\Phi_n(x)$ has degree $\phi(n)$ ($\phi$ denotes the Euler totient function).
    \item The coefficients of $\Phi_n(x)$ lie in $\mathbf{Z}$.
    \item $\Phi_n(x)$ is irreducible over $\mathbf{Q}$.
    \item If $p$ is prime, then $\Phi_p(x)=1+x+\dotsc +x^{p-1}$.
    \item $\Phi_{p^n}(x)=\Phi_p(x^{p^{n-1}})$.
\end{itemize}

A \emph{primitive element} or \emph{primitive root} of $\mathbf{Z}/n\mathbf{Z}$ is an element that generates the group of units $(\mathbf{Z}/n\mathbf{Z})^\times$. A \emph{primitive polynomial} over $\mathbf{F}_p$ is the minimal polynomial of some primitive element $\alpha\in\mathbf{F}_q$. The following result (see \cite[Theorem 2.47]{LN97} for example) is useful.
\begin{lemma}
\label{lem: cyclotomic irr}
Let $n$ be a positive integer relatively prime to a prime power $q$, and let $d$ be the order of $q$ modulo $n$. Since the cyclotomic polynomial $\Phi_n$ has coefficients in $\mathbf{Z}$, it can be viewed as a polynomial in $\mathbf{F}_q[x]$, and as such, it has $\phi(n)/d$ distinct irreducible factors, each of which has degree $d$.
\end{lemma}

From now on, we work over $\mathbf{F}_2$, and we observe that if $n$ is an odd integer and $2$ is a primitive root modulo $n$, then Lemma \ref{lem: cyclotomic irr} guarantees that the cyclotomic polynomial $\Phi_n$ is irreducible. For example, $1+x+x^2+x^3+x^4=\Phi_5(x)$ and $1+x^3+x^6=\Phi_9(x)$ are both irreducible over $\mathbf{F}_2$. 

\subsubsection{Trinomials}
A huge collection of explicit trinomials $x^n+x^m+1$ which are primitive irreducible over $\mathbf{F}_2$ is available, see \cite{BZ11} and \cite[Chapter 3 \S 5]{LN97}. Consider $x^n+x+1$. Computationally, they are often irreducible (but certainly not always). They have the following useful property however.

\begin{lemma}
For all $n\geq 2$, the polynomial $x^n+x+1$ has no repeated factors over $\mathbf{F}_2$.
\end{lemma}
\begin{proof}
A polynomial has repeated factors if and only if it shares a common factor with its formal derivative. If $n$ is even and $f(x)=x^n+x+1$, then $f'(x)=1$ and so $f'$ has no common factor with $f$. If $n$ is odd, $f'(x)=x^{n-1}+1$. Then if $r$ denotes a root of $f'(x)$ (in some splitting field), we have $r^{n-1}=1$, so $r^n=r$, and thus $f(r)=r^n+r+1=1$. It follows that $r$ is not a root of $f$, so $f$ and $f'$ cannot share any common factors.
\end{proof}

\subsection{Fourier analysis over \texorpdfstring{$\left(\mathbf{F}_2\right)^d$}{F\_2\^{}d}}
Let $\left(\mathbf{F}_2\right)^d$ be the abelian group of length $d$ binary vectors under coordinate-wise addition. The characters of $\left(\mathbf{F}_2\right)^d$ are indexed by $\beta\in\left(\mathbf{F}_2\right)^d$:
\begin{equation*}
    \chi_\beta(\alpha)=(-1)^{\alpha\cdot \beta},
\end{equation*}
where $\alpha\cdot \beta$ denotes the number of coordinates $i$ for which $\alpha_i=\beta_i=1$ (alternatively, it can be thought of as a dot product over $\mathbf{F}_2$).

If $Q(\alpha)$ is a probability distribution on $\left(\mathbf{F}_2\right)^d$ (or more generally any function $\left(\mathbf{F}_2\right)^d\to\mathbf{C}$), its \emph{Fourier transform} at $\beta\in\left(\mathbf{F}_2\right)^d$ is
\begin{equation*}
    \widehat{Q}(\beta)=\sum_{\alpha\in\left(\mathbf{F}_2\right)^d}Q(\alpha)(-1)^{\alpha\cdot \beta}.
\end{equation*}
It's easy to see that $\widehat{Q}(0)=1$. The uniform distribution $U(\alpha)=1/2^d$ for all $\alpha\in\left(\mathbf{F}_2\right)^d$ has the Fourier transform
\begin{align*}
    \widehat{U}(0)&=1
    \\\widehat{U}(\alpha)&=0,\qquad \alpha\neq 0.
\end{align*}
The \emph{convolution} of two probabilities $Q_1$, $Q_2$ is
\begin{equation*}
    (Q_1\ast Q_2)(\alpha)=\sum _{\gamma}Q_1(\gamma)Q_2(\alpha+\gamma).
\end{equation*}
Note that if $X_1$ and $X_2$ are independent random variables in $\left(\mathbf{F}_2\right)^d$ with distributions $Q_1$ and $Q_2$ respectively, then $X_1+X_2$ has $Q_1\ast Q_2$ as its distribution. The Fourier transform turns convolution into product, with
\begin{equation*}
    \widehat{Q_1\ast Q_2}(\beta)=\widehat{Q}_1(\beta)\widehat{Q}_2(\beta).
\end{equation*}
The measure $Q$ can be recovered from its Fourier transform via the inversion formula
\begin{equation*}
    Q(\alpha)=\frac{1}{2^d}\sum_{\beta}(-1)^{\alpha\cdot \beta}\widehat{Q}(\beta).
\end{equation*}
Finally, the Plancherel theorem relates the $L^2$ norm of $Q$ with $\widehat{Q}$, and states
\begin{equation*}
    2^d\sum_{\alpha\in\left(\mathbf{F}_2\right)^d}|Q(\alpha)|^2=\sum_{\beta\in(\mathbf{F}_2)^d}|\widehat{Q}(\beta)|^2
\end{equation*}

The following upper bound lemma is the key to establishing the upper bound in Theorem \ref{thm: main}. It is a direct consequence of the Plancherel theorem.
\begin{lemma}
\label{lem: upper bound}
Let $Q(\alpha)$ be a probability on $\left(\mathbf{F}_2\right)^d$ and let $U(\alpha)$ be the uniform distribution. Then
\begin{equation*}
    4\|Q-U\|_{TV}^2\leq 2^d\sum_{\alpha}(Q(\alpha)-U(\alpha))^2=\sum_{\beta\neq 0}|\widehat{Q}(\beta)|^2.
\end{equation*}
\end{lemma}
\begin{proof}
The inequality follows by Cauchy--Schwarz and the equality follows from the Plancherel theorem and the fact that $\widehat{U}(0)=\widehat{Q}(0)=1$ and $\widehat{U}(\alpha)=0$ for $\alpha\neq 0$.
\end{proof}

To set up the application of Lemma \ref{lem: upper bound} towards the proof of Theorem \ref{thm: main}, let $e_1, \dotsc, e_d$ be the standard basis for $\left(\mathbf{F}_2\right)^d$. Let
\begin{equation*}
    Q(\alpha)=\begin{cases}
    \frac{1}{2} & \alpha=0
    \\\frac{1}{2d} & \alpha=e_i
    \\0 & \text{else}
    \end{cases}.
\end{equation*}
Then
\begin{equation*}
    \widehat{Q}(\beta)=\sum_{\alpha}Q(\alpha)(-1)^{\alpha\cdot \beta}=\frac{1}{2}+\frac{1}{2d}\sum_{i=1}^d (-1)^{\beta_i}=1-\frac{|\beta|}{d},
\end{equation*}
where $|\beta|$ denotes the number of non-zero entries in $\beta$ (with respect to the standard basis).

Let $A:\left(\mathbf{F}_2\right)^d\to\left(\mathbf{F}_2\right)^d$ be a linear map, and consider the Markov chain starting from $X_0=0$, and
\begin{equation}
\label{eq: lin walk}
    X_n=AX_{n-1}+\varepsilon_n,
\end{equation}
with $\Prob(\varepsilon_n=\alpha)=Q(\alpha)$ for all $\alpha\in\left(\mathbf{F}_2\right)^d$, and the $\varepsilon_n$ independent. Iterating, $X_0=0$. $X_1=\varepsilon_1$, $X_2=A\varepsilon_1+\varepsilon_2$ and so on, and so
\begin{equation}
\label{eq: distr at time n}
    X_n=A^{n-1}\varepsilon_1+A^{n-2}\varepsilon_2+\dotsm+\varepsilon_n.
\end{equation}
As this is a sum of independent random variables, if $Q_n(\alpha)=\Prob(X_n=\alpha)$, then
\begin{equation}
\label{eq: fourier transform time n}
    \widehat{Q}_n(\beta)=\prod_{j=0}^{n-1}\left(1-\frac{|(A^t)^j\beta|}{d}\right).
\end{equation}
In our application, $A$ will be the matrix of squaring (which is linear in characteristic $2$), $A^d=I$, and the product becomes tractable.

In \cite{DG92a, DG92b}, this technique was used on $\left(\mathbf{F}_2\right)^d$ with
\begin{equation*}
    A=\left(\begin{array}{cccc}
        1 &  & & \\
        1 & 1&& \\
         & \ddots&\ddots& \\
         & & 1& 1
    \end{array}\right)
\end{equation*}
($1$'s along the diagonal and lower subdiagonal and $0$ otherwise) to get sharp results. See \cite{D88} for applications to non-Abelian groups.

The following proposition shows that adding deterministic mixing in this situation cannot slow things down. It gives one way of proving the upper bound in Theorem \ref{thm: main}.
\begin{proposition}
\label{prop: perm mixes no slower}
Let $A:\left(\mathbf{F}_2\right)^d\to\left(\mathbf{F}_2\right)^d$ be any invertible linear map, and consider the walk \eqref{eq: lin walk}. Let $P_n$ be the walk $X_n=X_{n-1}+\varepsilon_n$ without applying $A$. Then
\begin{equation*}
    \|Q_n-U\|_2^2\leq \|P_n-U\|_2^2,
\end{equation*}
where the $L^2$ norm is defined by
\begin{equation*}
    \|P-Q\|_2^2=2^d\sum_{\alpha\in\left(\mathbf{F}_2\right)^d}|P(\alpha)-Q(\alpha)|^2.
\end{equation*}
\end{proposition}
\begin{proof}
Note that
\begin{equation*}
    \|Q_n-U\|_2^2=\sum_{\beta\neq 0} \prod_{j=0}^{n-1}\left(1-\frac{|(A^t)^j\beta|}{d}\right)
    \leq \sum_{\beta\neq 0} \prod_{j=0}^{n-1}\left(1-\frac{|\beta|}{d}\right)=\|P_n-U\|_2^2,
\end{equation*}
where the middle inequality is an application of the rearrangement inequality \cite{R52}, noting that an invertible linear map acts as a permutation on the non-zero elements of $\left(\mathbf{F}_2\right)^d$ and all factors are non-negative.
\end{proof}
\begin{remark}
Proposition \ref{prop: perm mixes no slower} says that applying a deterministic bijection between steps of the random walk on the hypercube cannot slow the mixing of the Markov chain (at least in an $L^2$ sense). While this is not very helpful if the resulting chain is supposed to mix faster, squaring fails to speed up the mixing (see Remark \ref{rmk: no speedup}) and so Proposition \ref{prop: perm mixes no slower} gives one way of proving the upper bound in Theorem \ref{thm: main}.
\end{remark}

\section{Proof of Theorem \ref{thm: main}}
\label{sec: proof}
Throughout this section, $p$ is a prime such that $2$ is a primitive root in $\mathbf{F}_{p}$, and $d=p-1$. By Lemma \ref{lem: cyclotomic irr}, the cyclotomic polynomial $\Phi_p(x)=1+x+\dotsm+x^{d}$ is irreducible over $\mathbf{F}_2$. Represent $\mathbf{F}_{2^d}\cong \mathbf{F}_2[x]/(\Phi_p)$. The random walk defined by \eqref{eq: chain defn} with basis \eqref{eq: basis} can be represented as \eqref{eq: lin walk} with the basis $e_i=x^{i-1}$, with $A$ being the matrix of squaring with respect to this basis. We will index the rows and columns of matrices starting from $0$ rather than $1$, to match the exponents in the powers of $x$.

\begin{example}
Consider the case of $p=5$. The matrix $A$ representing the linear map $x\mapsto x^2$ on $\mathbf{F}_{16}$ (viewed as an $\mathbf{F}_2$-vector space) with respect to the standard basis $1$, $x$, $x^2$, $x^3$ is
\begin{equation*}
    A=\begin{blockarray}{ccccc}
    &1&x&x^2&x^3\\
    \begin{block}{c(cccc)}
    1&1&0&1&0\\
    x&0&0&1&1\\
    x^2&0&1&1&0\\
    x^3&0&0&1&0\\
    \end{block}
    \end{blockarray}
    \quad A^2=\begin{blockarray}{ccccc}
    &1&x&x^2&x^3\\
    \begin{block}{c(cccc)}
    1&1&1&0&0\\
    x&0&1&0&0\\
    x^2&0&1&0&1\\
    x^3&0&1&1&0\\
    \end{block}
    \end{blockarray}
    \quad
    A^3=\begin{blockarray}{ccccc}
    &1&x&x^2&x^3\\
    \begin{block}{c(cccc)}
    1&1&0&0&1\\
    x&0&0&1&1\\
    x^2&0&1&0&1\\
    x^3&0&0&0&1\\
    \end{block}
    \end{blockarray}
\end{equation*}
and $A^4=I$. Note that $A^j$ is a permutation matrix with one column replaced by a column of all ones. If this column is $j^*$, then $j^*=(p-1)/2^j$. The following result shows that this holds for all primes $p$ where $2$ is a primitive root in $\mathbf{F}_p$.
\end{example}
\begin{proposition}
Suppose that $\Phi_p(x)=1+x+\dotsm+x^d$ is irreducible over $\mathbf{F}_2$. Then the matrix of squaring $j$ times, $A^j$, $1\leq j\leq d-1$, with respect to the basis $1,x,\dotsc, x^{d-1}$, is a permutation matrix where the column $j^*=(p-1)/2^j$ (starting the indexing from $0$) is replaced by all ones.
\end{proposition}
\begin{proof}
Note that since $x^p-1=(x-1)(x^{p-1}+\dotsm +1)=0$ in $\mathbf{F}_{2^d}$, $x^i=x^j$ if $i=j\pmod{p}$. The matrix of squaring $j$ times, $A^j$, sends $x^i$ to $x^{2^ji}$ for all $i$.

Since $2$ is a primitive root modulo $p$, as $j$ goes from $1$ to $p-2$, $2^j$ runs over all elements of $\mathbf{F}_p^\times$ except $1$. If $2^ji=p-1\pmod{p}$, then $x^{2^ji}=x^{p-1}+\dotsm+1$ and otherwise, it's equal to some $x^k$ with $1\leq k\leq p-2$. This means each column except $j^*$ has exactly one non-zero entry, where it's $1$. Moreover, since $2^j$ is invertible modulo $p$, all rows can have at most one non-zero entry off the column $j^*$.
\end{proof}

Next, consider \eqref{eq: distr at time n} with $n=dm$ for some positive integer $m$. From \eqref{eq: fourier transform time n},
\begin{equation}
\label{eq: fourier transform dm}
    \widehat{Q}_n(\beta)=\prod _{j=0}^{d-1}\left(1-\frac{|(A^t)^j\beta|}{d}\right)^m.
\end{equation}
The next result determines these values.

\begin{proposition}
\label{prop: fourier transform}
Let $\beta_i$ denote the coefficient of $x^i$ in $\beta$. The Fourier transform of the square and add Markov chain, \eqref{eq: fourier transform dm}, after $n=dm$ steps satisfies $\widehat{Q}_{n}(\beta)=\widehat{Q}_d(\beta)^m$ and
\begin{equation*}
    \widehat{Q}_d(\beta)=\begin{cases}\left(1-\frac{|\beta|}{d}\right)^{d-|\beta|}\left(1-\frac{|\beta|-1}{d}\right)^{|\beta|}&|\beta|\text{ is even, }\beta_0=0
    \\\left(1-\frac{|\beta|}{d}\right)^{|\beta|+1}\left(1-\frac{|\beta|+1}{d}\right)^{d-|\beta|-1}&|\beta|\text{ is odd, }\beta_0=0
    \\\left(1-\frac{|\beta|}{d}\right)^{d-|\beta|+1}\left(1-\frac{|\beta|-1}{d}\right)^{|\beta|-1}&|\beta|\text{ is even, }\beta_0=1
    \\\left(1-\frac{|\beta|}{d}\right)^{|\beta|}\left(1-\frac{|\beta|+1}{d}\right)^{d-|\beta|}&|\beta|\text{ is odd, }\beta_0=1.
    \end{cases}
\end{equation*}
\end{proposition}
\begin{proof}
The key point is that the matrix $A^j$ is a permutation matrix except for one column of all ones. The all ones column $j^*$ occurs exactly once in the positions $1,2,\dotsc, d-1$ as $j$ varies in $\{1,2,\dotsc, d-1\}$. The argument then follows by considering the four separate cases.

For example, when $|\beta|$ is even and $\beta_0=0$, there are exactly $|\beta|$ many non-zero entries in the vector $\beta$ among the coefficients of $x,\dotsc, x^{d-1}$. When $j^*$ is among the indices where $\beta$ is non-zero, $(A^t)^j\beta$ has $1$ fewer non-zero entry (since one of the $1$'s was replaced by $\beta\cdot (1,\dotsc, 1)=0$). This occurs exactly $|\beta|$ many times. Otherwise, the number of non-zero entries remains the same. This gives the desired expression.

The other cases are similar.
\end{proof}

\begin{proof}[Proof of Theorem \ref{thm: main}]
From the upper bound lemma (Lemma \ref{lem: upper bound}), for $n=dm$,
\begin{equation}
\label{eq: upper bound}
    2^d\sum_{\alpha\in\mathbf{F}_{2^d}}|Q_n(\alpha)-U(\alpha)|^2=\sum_{\beta\neq 0}\widehat{Q}_d(\beta)^{2m}.
\end{equation}
For the four cases in Proposition \ref{prop: fourier transform}, the sum in \eqref{eq: upper bound} breaks into four sums:
\begin{equation}
\label{eq: four sums}
\begin{split}
    \Sigma_I&=\sum_{j\text{ even}}\left(1-\frac{j}{d}\right)^{2m(d-j)}\left(1-\frac{j-1}{d}\right)^{2mj}{d-1\choose j}
    \\\Sigma_{II}&=\sum_{j\text{ odd}}\left(1-\frac{j}{d}\right)^{2m(j+1)}\left(1-\frac{j+1}{d}\right)^{2m(d-j-1)}{d-1\choose j}
    \\\Sigma_{III}&=\sum_{j\text{ even}}\left(1-\frac{j}{d}\right)^{2m(d-j+1)}\left(1-\frac{j-1}{d}\right)^{2m(j-1)}{d-1\choose j-1}
    \\\Sigma_{IV}&=\sum_{j\text{ odd}}\left(1-\frac{j}{d}\right)^{2mj}\left(1-\frac{j+1}{d}\right)^{2m(d-j)}{d-1\choose j-1}.
\end{split}
\end{equation}

Let us use the expressions in \eqref{eq: four sums} to prove an $L^2$ lower bound. Because of the equality \eqref{eq: upper bound}, the $L^2$ norm is bounded below by any single term. Choose $j=2$ in $\Sigma_I$. This is
\begin{equation}
\label{eq: lower bound term}
    \left(1-\frac{2}{d}\right)^{2m(d-2)}\left(1-\frac{1}{d}\right)^{4m}{d-1\choose 2}=\left(1+\frac{1}{d-2}\right)^{4m}\left(1-\frac{2}{d}\right)^{2md}{d-1\choose 2}.
\end{equation}
Choose $m=\frac{1}{2}(\log(d)-c)$. For $d$ large,
\begin{equation*}
\begin{split}
    \left(1+\frac{1}{d-2}\right)^{4m}&=1+o(1)
    \\\left(1-\frac{2}{d}\right)^{2md}&\sim e^{-4m}=d^{-2}e^{2c}
    \\{d-1\choose 2}&\sim \frac{d^2}{2}.
\end{split}
\end{equation*}
Thus, the right hand side of \eqref{eq: lower bound term} is asymptotic to $e^{2c}/2$. It follows that $Q_n$ is exponentially far from uniform if $n=d(\log(d)-c)/2$. A similar argument shows, for this $n$, the total variation distance to uniform is exponentially close to $1$; this uses the (available) second moment method, see \cite[Proposition 7.14]{LP17}.

Proceed to the upper bound. By the upper bound lemma and Proposition \ref{prop: perm mixes no slower},
\begin{equation*}
    4\|Q_n-U\|_{TV}^2\leq \|P_n-U\|_2^2
\end{equation*}
where $P_n$ is the distribution of the random walk $X_n=X_{n-1}+\varepsilon_n$ on $\left(\mathbf{F}_2\right)^d$ after $n$ steps. It is known (see \cite{D88} for example) that if $n=\frac{1}{2}d(\log(d)+c)$,
\begin{equation*}
    \|P_n-U\|_2^2\leq e^{e^{-c}}-1
\end{equation*}
and $e^{e^{-c}}-1$ goes to zero like $e^{-c}$ when $c$ is large, which gives the desired upper bound.

\end{proof}

\begin{remark}
\label{rmk: no speedup}
Note that the random walk $X_n=X_{n-1}+\varepsilon_n$ on $\left(\mathbf{F}_2\right)^d$ without squaring also takes $\frac{1}{2}d(\log(d)+c)$ steps to equilibriate. Thus, in this case squaring does not introduce a dramatic speedup.
\end{remark}

\begin{remark}
The upper bound can also be proven directly from Proposition \ref{prop: fourier transform}. These more detailed calculations yield essentially the same answers as the rearrangement bounds.
\end{remark}

\begin{remark}
All the arguments given when $p=2$ extend to the case of a general prime $p$, with squaring replaced by taking the $p$th power. An upper bound on the $L^2$ distance needed to apply Proposition \ref{prop: perm mixes no slower} can be found in \cite{DSC96}, which would show that the Markov chain mixes after order $p^2d\log(d)$ steps. 
\end{remark}

\section{Back to squaring and adding on \texorpdfstring{$\mathbf{F}_p$}{F\_p}}
\label{sec: over F_p}
Return to our motivating problem
\begin{equation}
\label{eq: sq and add mod p}
    X_n=X_{n-1}^2+\varepsilon_n \pmod{p},
\end{equation}
where $p$ is a prime and $\varepsilon_n$ is $1$ or $-1$, independently, with probability $1/2$. To showcase the difference, consider the first problem: what is the stationary distribution of this Markov chain? Call this stationary distribution $\pi_p$.

A look at the data shows that for $p\geq 7$, there are many $j$ with $\pi_p(j)=0$ \emph{and} for some $p$, the non-zero $\pi_p(j)$ vary wildly in magnitude while for some $p$, $\pi_p(j)$ is roughly uniform.

The data below is normalized so that $\widetilde{\pi}_p$ is the left eigenvector for the eigenvalue $1$, scaled so all entries are integers.

\begin{example}[p=29]
\begin{equation*}
\begin{split}
    \widetilde{\pi}_{29}=(4, 2, 2, 2, 0, 8, 2, 6, 7, 0, 5, 0, 4, 0, 4, 0, 0, 0, 0, 3, 0, 5, 0, 2, 8, 0, 8, 2, 2).
\end{split}
\end{equation*}
\end{example}

\begin{example}[p=31]
\begin{equation*}
\begin{split}
    \widetilde{\pi}_{31}=(2, 3, 2, 4, 2, 2, 4, 2, 4, 4, 2, 2, 0, 2, 0, 4, 0, 4, 2, 4, 2, 2, 0, 0, 2, 0, 2, 2, 0, 2, 1).
\end{split}
\end{equation*}
Here (and in fact for all $p=3\pmod{4}$, see Theorem \ref{thm: 3 mod 4}), the smallest non-zero entry is $1$, the largest is $4$, and $1$ and $3$ only appear once.
\end{example}

\begin{example}[p=101]

\begin{equation*}
\begin{split}
    \widetilde{\pi}_{101}=(&66056, 33028, 33028, 33028, 0, 33028, 0, 0, 48868, 0, 48868, 0, 7376, 
    \\&48200, 7376, 62952, 21038, 14752, 21038, 0, 32951, 0, 68115, 0, 
    \\&85876, 0, 50712, 0, 0, 16514, 0, 16514, 34236, 0, 34236, 14752, 0, 
    \\&14752, 0, 0, 0, 0, 3688, 0, 34700, 0, 32856, 0, 3688, 0, 1844, 34236, 0, 
    \\&53012, 0, 26152, 0, 7376, 0, 0, 0, 0, 0, 33028, 0, 33028, 0, 27788,  0, 
    \\&62164, 51958, 34376, 51958, 0, 0, 18040, 0, 18040, 0, 68115, 0, 96465, 
    \\&0, 44864, 0, 16514, 7376, 0, 7376, 0, 0, 17188, 0, 17188, 3688, 29504, 
    \\&68396, 29504, 64708, 33028, 33028).
\end{split}
\end{equation*}
Here, the ratio of the largest to smallest non-zero entry is large ($\max/\min\doteq 52$). There appears to be unbounded fluctuation for larger $p$ with $p=1\pmod{4}$.
\end{example}

\begin{example}[p=103]
\begin{equation*}
\begin{split}
    \widetilde{\pi}_{103}=(&2, 3, 2, 4, 0, 2, 2, 2, 4, 2, 2, 0, 2, 2, 4, 4, 4, 4, 4, 2, 2, 0, 2, 0, 4, 2, 2, 4, 
    \\&2, 4, 2, 4, 2, 4, 2, 4, 0, 4, 0, 2, 2, 0, 2, 0, 0, 2, 0, 2, 2, 2, 2, 4, 0, 2, 2, 2, 
    \\&2, 4, 2, 4, 4, 2, 4, 2, 2, 4, 0, 4, 0, 2, 0, 2, 0, 2, 0, 2, 0, 2, 2, 0, 4, 2, 4, 2, 
    \\&2, 0, 0, 0, 0, 0, 2, 2, 4, 2, 2, 0, 2, 2, 2, 4, 0, 2, 1).
\end{split}
\end{equation*}
\end{example}

In \emph{all} cases we looked at, the Markov chain was ergodic (had a unique eigenvector with eigenvalue $1$). We are unable to prove this in general.

There is \emph{some} sense to be made: observe that if $j$ has both $j-1$ and $j+1$ non-squares modulo $p$, then $\pi_p(j)=0$. Classical number theory (see \cite[Chapter 5, Exc. 8]{IR90} for example) shows that asymptotically, this accounts for a quarter of all $j$. This matches the data when $p=3\pmod{4}$. For example, when $p=103$, $\pi_{103}(j)=0$ for 25 values of $j$. However, when $p=1\pmod{4}$, there are further forced zeroes, with $\pi_{101}(j)=0$ for 44 values of $j$.

Ron Graham and Steve Butler observed that:
\begin{itemize}
    \item When $p=3\pmod{4}$, these $j\pm 1$ non-residues \emph{exactly} matches the zeroes (for all $p\leq 10000$).
    \item When $p=1\pmod{4}$, the proportion of zeroes appears to be converging to approximately 42\%.
\end{itemize}
We record one further piece of mathematical progress, which explains the first point.

\begin{theorem}[He, \cite{H20}]
\label{thm: 3 mod 4}
If $p=3\pmod{4}$, then the square-and-add Markov chain \eqref{eq: sq and add mod p} is irreducible, aperiodic, and has a unique stationary distribution given by
\begin{equation*}
    \pi_p(j)=\frac{|\{k\in\mathbf{F}_p\mid k^2\pm 1=j\}|}{2p}.
\end{equation*}
\end{theorem}

\section*{Acknowledgements}
We thank Amol Aggarwal, Steve Butler, Ron Graham, Bob Guralnick, David Kazhdan, Laurent Miclo, Yuval Peres and Kannan Soundararajan for their help. The first author was partially supported by NSF grant DMS 1954042. The second author was partially supported by NSERC.

\bibliography{bibliography}{}
\bibliographystyle{amsplain}

\appendix
\section{Alternative proof of upper bound in Theorem \ref{thm: main}}
\label{app: alt pf}
In this appendix, we give another proof of the upper bound in Theorem \ref{thm: main} directly from the eigenvalues computed in Proposition \ref{prop: fourier transform}, rather than using Proposition \ref{prop: perm mixes no slower}. We use the same setup and notation as defined previously.

Each of the four sums $\Sigma_I$, $\Sigma_{II}$, $\Sigma_{III}$ and $\Sigma_{IV}$ must be bounded. Throughout, $m=\frac{1}{2}(\log(d)+c)$. The bounds $(1-x)\leq e^{-x}$ and ${n\choose j}\leq n^j/j!$ are used.

\begin{itemize}
\item{$\Sigma_I$:}
A term in the sum is, with $j$ even,
\begin{equation*}
    \left(1-\frac{j}{d}\right)^{2m(d-j)}\left(1-\frac{j-1}{d}\right)^{2mj}{d-1\choose j}\leq \frac{e^{2mj(1-1/d)+j\log(d)}}{j!}.
\end{equation*}
The exponent on the right hand side is equal to $-cj(1-1/d)+2\log(d)/d$ and so
\begin{equation*}
    \Sigma_I\leq \sum_{j=1}^\infty \frac{e^{-cj(1-1/d)+2}}{j!}\leq f_I(c)
\end{equation*}
for some explicit $f_I(c)$ tending to $0$ when $c$ is large.

\item{$\Sigma_{II}$:}
A term in the sum is, with $j$ odd,
\begin{equation*}
\begin{split}
    \left(1-\frac{j}{d}\right)^{2m(j+1)}\left(1-\frac{j+1}{d}\right)^{2m(d-j-1)}{d-1\choose j}&\leq \left(1-\frac{j}{d}\right)^{2md}\frac{d^j}{j!}
    \\&\leq \frac{e^{-2mj+j\log(d)}}{j!}.
    \end{split}
\end{equation*}
The exponent equals $-cj$ and so the sum is bounded y an explicit $f_{II}(c)$ which tends to $0$ when $c$ is large.
\item{$\Sigma_{III}$:}
A term in the sum is, with $j$ even,
\begin{equation*}
\begin{split}
    \left(1-\frac{j}{d}\right)^{2m(d-j+1)}\left(1-\frac{j-1}{d}\right)^{2m(j-1)}{d-1\choose j-1}&\leq \left(1-\frac{j}{d}\right)^{2md}\frac{d^{j-1}}{(j-1)!}
    \\&\leq \frac{e^{-2m(j-1)+(j-1)\log(d)}}{(j-1)!}.
\end{split}
\end{equation*}
Again, the sum is bounded above by $f_{III}(c)$ which tends to $0$ for large $c$.
\item{$\Sigma_{IV}$:}
A term in the sum is, with $j$ odd,
\begin{equation*}
\begin{split}
    \left(1-\frac{j}{d}\right)^{2mj}\left(1-\frac{j+1}{d}\right)^{2m(d-j)}{d-1\choose j-1}&\leq \left(1-\frac{j}{d}\right)^{2md}\frac{d^{j-1}}{(j-1)!}
    \\&\leq \frac{e^{-2mj+(j-1)\log(d)}}{(j-1)!}.
\end{split}
\end{equation*}
This sums as before to $f_{IV}(c)$ tending to $0$ for large $c$.

\end{itemize}

Combining the bounds proves the upper bound claimed for Theorem \ref{thm: main}.

\end{document}